\documentclass[a4paper,11pt]{amsart}
\usepackage{amsthm,amssymb,latexsym, amsmath}
\input amssym.def

\newtheorem{theorem}{T{\hskip 0pt\footnotesize\bf HEOREM}}[section]
\newtheorem{lemma}[theorem]{L{\hskip 0pt\footnotesize\bf EMMA}}
\newtheorem{proposition}[theorem]{P{\hskip 0pt\footnotesize\bf ROPOSITION}}
\newtheorem{definition}[theorem]{D{\hskip 0pt\footnotesize\bf EFINITION}}
\newtheorem{corollary}[theorem]{C{\hskip 0pt\footnotesize\bf OROLLARY}}

\def\O{\Omega}

\def\D{\mathcal D}

\def\M{\mathcal{M}_d}
\def\RR{\mathcal{R}}
\def\la{\lambda}

\def\Da{\Delta}

\def\veps{\varepsilon}

\def\veps{\vec{\varepsilon}}

\def\vba{\vec{\beta}}

\def\vda{\vec{\delta}}

\def\BMO{\mathrm{BMO}}
\def\LMO{\mathrm{LMO}}
\def\bmo{\mathrm{bmo}}

\def\veps{\vec{\varepsilon}}

\def\vba{\vec{\beta}}

\def\vj{\vec{j}}
\def\vk{\vec{k}}
\newcommand{\bprop} {\begin{proposition}}
\newcommand{\eprop} {\end{proposition}}
\newcommand{\btheo} {\begin{theorem}}
\newcommand{\etheo} {\end{theorem}}
\newcommand{\blem} {\begin{lemma}}
\newcommand{\elem} {\end{lemma}}
\newcommand{\bcor} {\begin{corollary}}
\newcommand{\ecor} {\end{corollary}}

\newcommand{\Be}{\begin{equation}}
\newcommand{\Ee}{\end{equation}}
\newcommand{\Bea}{\begin{eqnarray}}
\newcommand{\Eea}{\end{eqnarray}}
\newcommand{\Bes}{\begin{equation*}}
\newcommand{\Ees}{\end{equation*}}
\newcommand{\Beas}{\begin{eqnarray*}}
\newcommand{\Eeas}{\end{eqnarray*}}
\newcommand{\Ba}{\begin{array}}
\newcommand{\Ea}{\end{array}}
\newcommand{\mult}{\mathrm{Mult}}
\newcommand{\empset}{\emptyset}
\def\R{\mathbb{R}}

\def\N{\mathbb{N}}

\def\T{\mathbb{T}}

 \scrollmode

\title{From little $\BMO$ to product $\BMO$ through multiplication operator}

\author{Benoit F. Sehba}
\thanks{Department of Mathematics, School of Physics and Mathematical Sciences, University of Ghana, P.O. Box LG 62, Legon-Accra, Ghana. {\tt bfsehba@ug,edu,gh} }

\begin{document}
\maketitle
\begin{abstract} We characterize the multipliers from the little $\BMO$ of Cotlar-Sadosky to the product $\BMO$ of Chang-Fefferman  on the polydisk. 
\end{abstract}
\section{Introduction and notation}
Recall that in the one parameter setting, the space of functions of bounded mean oscillation on the torus $\T$, $\BMO(\T)$ consists of all measurable functions $f$ on $\T$ such that
\begin{equation}
\|f\|_*:=\sup_{I\subset \T}\frac{1}{|I|}\int_{I}|f(t)-m_If|dt<\infty,
\end{equation}
where the supremum is taking on all intervals $I\subset \T$, and $m_If=\frac{1}{|I|}\int_I f(s)ds$. Here, for any measurable set $E$, $|E|$ denotes the Lebesgue measure of $E$. 
\medskip

The multiplier algebra on $\BMO(\T)$ has been characterized by D. Stegenga in \cite{stegenga}. He proved that a function $b$ is a multiplier of $\BMO(\T)$ if and only if $b$ is bounded and satisfies the condition
\begin{equation*}
\sup_{I\subset \T}\frac{\log\left(|I|\right)}{|I|}\int_{I}|f(t)-m_If|dt<\infty,
\end{equation*}

When comes to the multiparameter setting, there appear three main notions of bounded mean oscillation. Our interest here is for multipliers between the generalization of the little space of functions of bounded mean oscillation introduced by M. Cotlar and C. Sadosky in \cite{cotsad} and the product space of functions of bounded mean oscillation on open subset of $\T^N$ discovered by A. Chang and R. Fefferman \cite{ChFef1}.


Recall that the so-called small BMO space on the bitorus $\T^2$, introduced by Cotlar and Sadosky, is defined as
\begin{equation} \label{bmo1}
   \bmo(\T^2):= \{ b \in L^2(\T^2): \sup_{R \subset \T^2 \text{ rectangle }} \frac{1}{|R|}\int_R |b(s,t) - m_R b| ds dt < \infty \},
\end{equation}
$m_Rb=\frac{1}{|R|}\int_R b(s,t) ds dt.$ It easily generalizes as follows,
\begin{equation} \label{bmo}
   \bmo(\T^N)= \{ b \in L^2(\T^N): ||b||_{*,N}:=\sup_{R \subset \T^N \text{ rectangle }} \frac{1}{|R|}\int_R |b(t_1,\cdots,t_N) - m_R b| dt_1\cdots dt_N < \infty \},
\end{equation}
$m_Rb=\frac{1}{|R|}\int_Rb(t_1,\cdots,t_N)dt_1\cdots dt_N$. Seen as a quotient space with the set of constants, $\bmo(\T^N)$ is a Banach space with norm
$$\Vert b\Vert_{\bmo}:=||b||_{*,N}.$$
We observe that a function $b\in \bmo(\T^N)$ if and only if for any positive integers $N_1$ and $N_2$ such that $0<N_1+N_2=N$,
\begin{equation}\label{eq:equivdefbmo}
\max\{\sup_{s\in \T^{N_2}}\|b(\cdot,s)\|_{*,N_1},\sup_{t\in \T^{N_1}}\|b(t,\cdot)\|_{*,N_1}\}=M<\infty
\end{equation}
and $M$ is comparable to $||b||_{*,N}$ (see \cite{benoit1}). Note also that if $b\in \bmo(\T^N)$, then for any rectangle $R\subset \T^{N_1}$, $0<N_1<N$, $m_Rb\in \bmo(\T^{N-N_1})$ uniformly.
\vskip .2cm
The multi-parameter space of functions of bounded mean oscillation on open sets, or product $\BMO$ space on $\T^N$, was introduced by Chang and Fefferman in \cite{ChFef1}. It is the dual of the product
Hardy space $H^1(\T^N)$, given by
$$
H^1(\T^N):=\{f\in L^1(\T^N): H_jf\in L^1(\T^N), H_{k_1}\cdots H_{k_j}f\in L^1(\T^N),\,\,\,k_l\in \{1,\cdots, n\},\,\,\, j,l=1,\cdots ,n\},
$$
where $H_j$ is the Hilbert transform in the $j$-th coordinate. The space $\bmo(\T^N)$ is a strict subspace of $\BMO(\T^N)$ and they coincide only when $N=1$ (\cite{fergsad}).

Given two Banach spaces of functions $X$ and $Y$, the space of pointwise multipliers from $X$ to $Y$ is defined as follows
$$\mathcal {M}(X,Y)=\{\phi:\phi f\in Y\,\,\,\textrm{for all}\,\,\,f\in X\}.$$
When $X=Y$, we simply write $\mathcal {M}(X,X)=\mathcal {M}(X)$. When $\phi\in \mathcal {M}(X,Y)$, we say $\phi$ is a multipler from $X$ to $Y$.

The author obtained in \cite{benoit1} that the only multipliers of  $\bmo(\T^N)$ are the constants. The algebra of multipliers on $\BMO(\T^N)$ is considered by S. Pott and the author in a forthcoming paper (see \cite{benoit} for the two parameter case). In this note, we focus on the multipliers from $\bmo(\T^N)$ to $\BMO(\T^N)$. To deal with this question, we observe that it is enough to consider the same question for the dyadic versions of $\bmo(\T^N)$ and $\BMO(\T^N)$ namely $\bmo^d(\T^N)$ and $\BMO^d(\T^N)$. Once the multipliers from $\bmo^d(\T^N)$ to $\BMO^d(\T^N)$ are obtained, usual arguments ( \cite{pipher,treil}) allow to conclude for the multipliers from $\bmo(\T^N)$ to $\BMO(\T^N)$.

\section{Notations, definitions and statement of the result}
 We identify $\T$ with the interval $[0,1)$ in the usual way and write $\mathcal D$
for the set of all dyadic subintervals. We denote by
$\mathcal R=\D^N$ the set of all dyadic rectangles $R=I_1\times\cdots\times I_N$, where $I_j\in \mathcal D$,
$j=1,\cdots, N$. Let $h_I$ denote the Haar wavelet adapted
to the dyadic interval $I$,
$$
h_I=|I|^{-1/2}(\chi_{I^+}-\chi_{I^-}),
$$
where $I^+$ and $I^-$
are the right and left halves of $I$, respectively.

For any rectangle $R\in \mathcal {R}$, the product Haar wavelet
adapted to $R=I_1\times\cdots\times I_N$ is defined by $h_R(t_1,\cdots, t_N)=h_{I_1}(t_1)\cdots h_{I_N}(t_N)$.
These wavelets form an orthonormal basis of
$$
L_0^2(\T^N) = \left\{ f \in L^2(\T^N):  \int_\T f(\cdots,t_j,\cdots) dt_j =0, j=1,\cdots, N \text{ for a.e. }t_1,\cdots, t_N \in \T \right\},
$$
with
$$f=\sum_{R\in \mathcal {R}} \langle f,h_R \rangle h_R=\sum_{R\in \mathcal
{R}}f_Rh_R \qquad (f \in L^2_0(\T^N)).
$$
The mean of $f\in
L^2(\mathbb {T}^N)$ over the dyadic rectangle $R$ is denoted $m_Rf$, and $f_R$ stands for the Haar coefficient $\langle f, h_R \rangle$ of $f$. If $R=S\times T$, we will also write $f_R=f_{S\times T}$ when this is necessary.

The space of functions of dyadic bounded mean oscillations in $\mathbb T^N$,
$\BMO^d(\mathbb T^N)$, is the space of all function $f\in L_0^2(\mathbb
T^N)$ such that
\begin{equation}
 ||f||_{\BMO^d}^2:=\sup_{\O\subset \mathbb
{T}^N}\frac{1}{|\O|}\sum_{R\in \O}|f_R|^2=\sup_{\O\subset \mathbb
{T}^N}\frac{1}{|\O|}||P_\O f||_2^2 < \infty,
\end{equation}
where the supremum
is taken over all open sets $\O\subset \mathbb T^N$ and $P_\O$ is the
orthogonal projection on the subspace spanned by Haar functions
$h_R$, $R\in \mathcal R$ and $R\subset  \O$.
\vskip .2cm
Recall that $\BMO^d(\T^N)$ is the dual space of the dyadic product Hardy space $H^{1}_d(\mathbb T^N)$ defined by

$$
H^{1,d}(\mathbb T^N)=\{f\in L_0^1(\mathbb T^N): \mathcal {S}[f]\in
L^1(\mathbb T^N)\},
$$
where $\mathcal S$ is the dyadic square function,
\begin{equation}   \label{eq:sq}
     \mathcal {S}[f] = \left(\sum_{R \in \RR}  \frac{\chi_R}{|R|} |f_R|^2 \right)^{1/2}.
\end{equation}


We note that $\bmo^d(\T^N)$ is defined by taking supremum only on dyadic rectangles in the definition of $\bmo(\T^N)$.

We also introduce the mean dyadic product $\BMO$ that we denote $\BMO_m^d$, as the space of those functions $f\in L_0^2(\T^N)$ such that there is a constant $C>0$ such that for any integer $0<N_1<N$, and any dyadic rectangle $R\in \D^{N_1}$, $$\|m_Rf\|_{\BMO^d(\T^{N-N_1})}\le C.$$
We denote by $\|f\|_{\BMO_m^d(\T^N)}$ the infimum of the constants as above. We remark that for any $f\in \bmo^d(\T^N)$,
$$\|f\|_{\BMO_m^d(\T^{N})}\le \|f\|_{\bmo^d(\T^{N})}.$$

As usual, for $\vj =(j_1,\cdots, j_N) \in \N_0 \times\cdots\times \N_0=\N_0^N$ we define the $j_1$-th generation of dyadic intervals and the $\vj$-th generation of dyadic rectangles,
$$
     \D_{j_1} = \{I \in \D: |I|= 2^{-j_1} \},
$$
$$
    \RR_{\vj} = \D_{j_1} \times\cdots\times \D_{j_N} = \{ I_1 \times\cdots\times I_N \in \RR: |I_k|= 2^{-j_k}\}.
$$

%
The product Haar martingale differences are given by
$$
   \Da_{\vj} f = \sum_{R\in \RR_{\vj}} \langle  f, h_R \rangle h_R,
$$
and the expectations by
$$
   E_{\vj} f = \sum_{\vk \in \N_0^N ,  \vk < \vj } \Da_{\vk}f,
$$
where we write $ \vk < \vj $ for $k_l < j_l$, $l=1,\cdots,N$ and correspondingly $\vk \le \vj $ for $k_l \le j_l$,
for $f \in L^2(\T^N)$.

The following operators defined on $L^2(\T^N)$ will be also needed
\begin{equation}   \label{eq:qdef}
   Q_{\vj}f =\sum_{ \vk \ge \vj } \Da_{\vk}f.
\end{equation}
Let us recall the definition of the space of functions of dyadic logarithmic mean oscillation on the $\T^N$, $\LMO^d(\T^N)$ as introduced in \cite{sehba}.
\begin{definition}   \label{def:LMOd}
Let $\phi \in L^2(\T^N)$. We say that $ \phi \in \LMO^d(\T^N)$, if there exists $C >0$ with
$$
    \|Q_{\vj} \phi\|_{\BMO^d(\T^N)} \le C \frac{1}{\left(\sum_{k=1}^N  j_k\right) +N}
$$
for all $\vj = (j_1,\cdots, j_N)  \in \N_0^N$. The infimum of such constants is denoted by $\|\phi\|_{\LMO^d}$.
\end{definition}

The above space has the following equivalent definition (see \cite{sehba}).
\begin{proposition} \label{prop:LMOequivchar}
Let $\phi \in L^2(\T^N)$. Then $ \phi \in \LMO^d(\T^N)$, if and only if there exists a constant $C >0$ such that for each
dyadic rectangle $R= I_1 \times\cdots\times I_N$ and each open set $\Omega \subseteq R$,
\begin{equation}   \label{eq:lmochar}
   \frac{\left(\log\frac{4}{|I_1|}+\cdots +\log\frac{4}{|I_N|}\right)^2}{|\Omega|} \sum_{Q \in \RR, Q \subseteq \Omega} |\phi_Q|^2 \le C.
\end{equation}
\end{proposition}

Our strong version of the above space is defined as follows.
\begin{definition}\label{def:LMO}
If we write  $\LMO^{d,\alpha}(\T^N)$ for the space given in Definition  \ref{def:LMOd} but with respect to the translated grid by $\alpha\in [0,1)^N$,   $\mathcal {D}^{N,\alpha}$ of $\D^N$, then $\LMO(\T^N)$ is the intersection of all $\LMO^{d,\alpha}(\T^N)$.
\end{definition}

This property is well known for the product $\BMO$ (see \cite{pipher, treil}). We denote by $\|f\|_{\LMO(\T^N)}$, the infimum of the constants $C>0$ such that for any $\alpha\in [0,1)^N$, $$\|f\|_{\LMO^{d.\alpha}}<C.$$

It is clear that $\bmo(\T^N)$ embeds continuously in each $\bmo^{d,\alpha}(\T^N)$, $\alpha\in [0,1)^N$. The converse holds in the one-parameter case (see \cite{pipher}). Let us suppose that the converse holds in $(N-1)$-parameter case and prove that in this case, it also holds in $N$-parameter. If $f\in \cap_{\alpha\in [0,1)^N}\bmo^{d,\alpha}(\T^N)$, then  from what has been said at the beginning of the previous section, we have that for any $s\in \T$, the function $f(s,\cdot)$ is uniformly in $\bmo^{d,\beta}(\T^{N-1})$ for any $\beta\in [0,1)^{N-1}$. It follows from the hypothesis that $f(s,\cdot)$ is uniformly in $\bmo(\T^{N-1})$. Also for any $t\in \T^{N-1}$, the function $f(\cdot, t)$ is uniformly in $\BMO^{d,\alpha_1}(\T)$ for any $\alpha_1\in [0,1)$, hence $f(\cdot, t)$ is uniformly in $\cap_{\alpha_1\in [0,1)}\BMO^{d,\alpha_1}(\T)=\BMO(\T)$. We conclude that that $f\in \bmo(\T^N)$.
\vskip .2cm
Our main result in this paper is the following.
\begin{theorem}\label{main}
$\mathcal M(\bmo(\T^N), \BMO(\T^N))=L^\infty(\T^N)\cap \LMO(\T^N).$
\end{theorem}

To prove the above theorem, we decompose the product of two functions in the Haar basis of $L_0^2(\T^N)$ and study the boundedness of each of the operators appearing in this decomposition. The boundedness of paraproducts on $L^2(\T^N)$ will be a useful tool among others.
\vskip .2cm
All over the text, given two positive quantities $A$ and $B$, the notation $A\lesssim B$ means that there exists some constant $C>0$ such that $A\leq CB$. When $A\lesssim B$ and $B\lesssim A$, we write $A \simeq B$.

\section{Few useful facts}
\subsection{Little $\BMO$, product $\BMO$ and a useful estimate}
We start by recalling the following embedding (see \cite{cotsad}).
\begin{lemma}\label{lem:bmoBMOembeddings}
There exists a constant $C>0$ such that for any $f\in \bmo^d(\T^N)$,
\begin{equation}\label{eq:bmoBMOembeddingscasegene}
\|f\|_{\BMO^d(\T^N)}\le C\|f\|_{\bmo^d(\T^N)}.
\end{equation}
\end{lemma}

Note that it was proved in \cite{bp} that
for any integers $0<N_1,N_2<N$ with $N_1+N_2=N$, and for any $f\in \BMO^d(\T^{N_1})$ and $g\in \BMO^d(\T^{N_2})$, the tensor product $f\otimes g$ defined by $$f\otimes g(s,t)=f(s)g(t),\,\,\,s\in \T^{N_1},\,\,\,t\in \T^{N_2}$$ is uniformly in $\BMO^d(\T^N)$ with
$$\|f\otimes g\|_{\BMO(\T^N)}\lesssim \|f\|_{\BMO(\T^{N_1})}\|g\|_{\BMO(\T^{N_2})}.$$

Hence the following result is enough to conclude that the embedding $\bmo^d(\T^N)\hookrightarrow \BMO^d(\T^N)$ is strict (see also \cite{cotsad}).
\begin{lemma}
Let $0<N_1,N_2<N$ be integers with $N_1+N_2=N$. Then for any $f\in \bmo^d(\T^{N_1})$ and $g\in \bmo^d(\T^{N_2})$, the tensor product $f\otimes g$ defined by $$f\otimes g(s,t)=f(s)g(t),\,\,\,s\in \T^{N_1},\,\,\,t\in \T^{N_2}$$ is in $\bmo^d(\T^N)$ if and only if $f$ and $g$ are both bounded. i.e $f\in L^\infty(\T^{N_1})$ and $g\in L^\infty(\T^{N_2})$.
\end{lemma}
\begin{proof}
Of course if $f\in L^\infty(\T^{N_1})$ and $g\in L^\infty(\T^{N_2})$, then $f\otimes g\in \bmo^d(\T^N)$ uniformly. Let us now suppose that $f\in \bmo^d(\T^{N_1})$ and $g\in \bmo^d(\T^{N_2})$ are such that $f\otimes g\in \bmo^d(\T^N)$. We may suppose without loss of generality that $\|f\|_{\bmo^d(\T^{N_1})}=\|g\|_{\bmo^d(\T^{N_2})}=1$. Recalling that if $h\in \bmo^d(\T^N)$, then for any $R\in \D^{N'}$, $0<N'<N$, $\|m_Rh\|_{\bmo^d(\T^{N-N'})}\le \|h\|_{\bmo^d(\T^N)}$, we obtain that for any $S\in \D^{N_1}$,
$$|m_Sf|\|g\|_{\bmo^d(\T^{N_2})}=\|m_S\left(f\otimes g\right)\|_{\bmo^d(\T^{N_2})}\le \|f\otimes g\|_{\bmo^d(\T^N)},$$
that is for any $S\in \D^{N_1}$,
$$|m_Sf|\le \|f\otimes g\|_{\bmo^d(\T^N)}<\infty.$$
It follows from the Lebesgue Differentiation Theorem that $f\in L^\infty(\T^{N_1})$. The same reasoning show that for any $T\in \D^{N_2}$,
$$|m_Tg|\le \|f\otimes g\|_{\bmo^d(\T^N)}<\infty.$$ Hence that $g\in L^\infty(\T^{N_2})$. The proof is complete.
\end{proof}
We will need the follows facts (see \cite{benoit1}).
\begin{lemma}\label{lem:testfunction}
The following assertions hold.
\begin{itemize}
\item[(i)] Given an interval $I\subset \T$, one can construct a function $\log_I$ such that
$$\log_I(t)\simeq \log\frac{4}{|I|}\,\,\,\textrm{for}\,\,\,t\in I$$ and $$\|\log_I\|_{\BMO(\T)}\le C<\infty$$
where the constant $C$ does not depend on $I$.
\item[(ii)] Given $f_j\in \BMO(\T)$, $j=1,\cdots, N$, the function $f(t_1,\cdots,t_N)=\sum_{j=1}^Nf_j(t_j)$ is uniformly in $\bmo(\T^N)$ with
$$\|f\|_{\bmo(\T^N)}\le \sum_{j=1}^N\|f_j\|_{\BMO(\T)}.$$
\end{itemize}
\end{lemma}
The following will be useful in the text.
\begin{lemma}\label{lem:estim}
There exists a constant $C>0$ such that for any dyadic rectangle $S\in \D^{N_1}$, any open set $\Omega\subseteq \T^{N_2}$, $N_1+N_2=N$, and any $b\in \bmo^d(\T^N)$, the following holds.
$$\|\chi_SP_\Omega b\|_{L^2(\T^N)}^2\le C|S||\Omega|\|b\|_{\bmo^d(\T^N)}.$$
\end{lemma}
\begin{proof}
We easily obtain
\begin{eqnarray*}
\|\chi_SP_\Omega b\|_{L^2(\T^N)}^2 &=& \|\chi_S\|P_\Omega b\|_{L^2(\T^{N_2})}^2\|_{L^2(\T^{N_1})}^2\\ &\leq& |\Omega|\|\chi_S\| b\|_{BMO(\T^{N_2})}^2\|_{L^2(\T^{N_1})}^2\\ &\leq& |\Omega|\| b\|_{bmo(\T^{N})}^2\|\chi_S\|_{L^2(\T^{N_1})}^2\\ &=& |S||\Omega|\|b\|_{\bmo^d(\T^N)}.
\end{eqnarray*}
\end{proof}
\subsection{Product of two functions in the Haar basis}
As previously said, we would like to start by characterizing multipliers from $\bmo(\T^N)$ to $\BMO(\T^N)$ on the dyadic side. For this, we will need to split the multiplication operator by $\phi$, $M_\phi$ into several operators that we will then estimate. We start by few definitions and observations.
\vskip .2cm

For $I$ a dyadic interval and $\varepsilon\in \{0,1\}$, we define $h^{\varepsilon}_I$ by
$$h^{\varepsilon}_I =\left\{ \begin{matrix} h_I &\text{if }& \varepsilon=0\\
      |I|^{-1/2}|h_I| & \text{ if } & \varepsilon=1
                                  \end{matrix} \right.
$$

For $R=I_1\times\cdots\times I_N\in \RR$ and $\vec {\varepsilon}=(\varepsilon_1,\cdots,\varepsilon_N)$, with $\varepsilon_j\in \{0,1\}$, we write
$$h^{\vec {\varepsilon}}_R(t1,\cdots,t_N)=h^{\varepsilon_1}_{I_1}(t_1)\cdots h^{\varepsilon_N}_{I_N}(t_N).$$
We recall that the sequence $\{h_R\}_{R\in \mathcal R}$ forms an orthonormal basis of $L_0^2(\T^N)$. We observe that given $f,\phi\in L_0^2(\T^N)$ with finite Haar expansions, we can write $$f(s,t)=\sum_{I\in \D}f_I(s)h_I(t),\,\,\,s\in \T^{N-1},\,\,\,t\in \T $$
and $$\phi(s,t)=\sum_{I\in \D}\phi_I(s)h_I(t),\,\,\,s\in \T^{N-1},\,\,\,t\in \T. $$
It follows from the usual one-parameter expansion that
\begin{equation}\label{eq:expand1}\left(f\phi\right)(s,t)=\Pi_\phi^1f(s,t)+\Delta_\phi^1 f(s,t)+\Pi_f^1\phi(s,t)
\end{equation}
where $$\Pi_\phi^1f(s,t)=\sum_{I\in \D}\phi_I(s)m_If(s)h_I(t),$$
$$\Delta_\phi^1 f(s,t)=\sum_{I\in \D}\phi_I(s)f_I(s)\frac{\chi_I(t)}{|I|}$$
and $$\Pi_f^1\phi (s,t)=\sum_{I\in \D}f_I(s)m_I\phi (s)h_I(t),\,\,\,s\in \T^{N-1},\,\,\,t\in \T.$$
Now observe that the coefficients in each of the above operators are product of two functions of $(N-1)$ variables, we repeat exactly the previous process for each of the coefficients. For example in $\Pi_\phi^1f$, let us expand $\phi_I(s)m_If(s)$. We obtain writing $s=(u,v),\,\,\,u\in \T^{N-2},\,\,\,v\in \T$ that
\Beas
\phi_I(u,v)m_If(u,v) &=& \sum_{J\in \D}\phi_{I\times J}(u)m_{I\times J}f(u)h_J(v)+ \sum_{I\in \D}\phi_{I\times J}(u)f_{I\times J}(u)\frac{\chi_J(v)}{|J|}\\ &+& \sum_{J\in \D}m_If_J(u)m_J\phi_I (u)h_J(v).
\Eeas
If we do the same for $\Delta_\phi^1 f$ and $\Pi_f^1\phi$, we obtain exactly that the product $f\phi$ is a sum of nine operators whose coefficients are product of two functions of $(N-2)$ variables. Repeating the same process for each of the coefficients, until we get only constant coefficients, we obtain that $f\phi$ can be written as a sum of $3^N$ operators.
\vskip .2cm
To describe the general form of the operators obtained, we consider the operators $B_{\veps,\vda,\vba}(\phi,\cdot)$ defined by
\begin{equation}\label{paraprodgene2}B_{\veps,\vda,\vba}(\phi,f):=\sum_{R\in \RR}\langle \phi,h^{\vec {\varepsilon}}_R\rangle \langle f,h^{\vec {\delta}}_R\rangle h^{\vec {\beta}}_R.\end{equation}

We denote by $\mathcal U$ the set of the triples $(\veps,\vda,\vba)$ satisfying the following relation:
\begin{itemize}
\item[1)] $(\veps,\vda,\vba)\neq (\vec {0},\vec {0},\vec {0})$,
\item[2)] $(\veps,\vda)\neq (\vec {1}, \vec {1})$,
\item[3)] $\epsilon_j+ \delta_j=1-\beta_j$.
\end{itemize}
here $\vec {0}=(0,\cdots, 0)$, and $\vec {1}=(1,\cdots, 1)$. We have that
\begin{equation}\label{eq:expanmuloper}f\phi=\sum_{(\veps,\vda,\vba)\in \mathcal {U}}B_{\veps,\vda,\vba}(\phi,f).
\end{equation}

Observe that for $\veps=\vec {0}$, the operators $B_{\vec {0},\vda,\vba}(\phi,)=\Pi_\phi^{\vba}$ are the paraproducts studied in \cite{sehba}. In particular, the operator $$\Pi_\phi^{\vec {0}}f=\Pi_\phi f:=\sum_{\vec {j}\in {\N_0}^N}(\Delta_{\vec j}\phi)(E_{\vec j}f)=\sum_{R\in \mathcal R}h_R\phi_Rm_Rf$$ is usually called the main paraproduct. It is easy to check that $\Pi_\phi$ is bounded on $L^2(\T^N)$ if and only if $\phi\in \BMO^d(\T^N)$. The other paraproducts are
$$\Delta_\phi f=\Pi_\phi^{\vec 1}f:=\sum_{\vec {j}\in {\N_0}^N}(\Delta_{\vec j}\phi)(\Delta_{\vec j}f)=\sum_{R\in \mathcal R}\frac{\chi_R}{|R|}\phi_Rf_R;$$
and for $\vec {\beta}\neq \vec {0}, \vec {1}$,
$$\Pi_\phi^{\vba}f:=\sum_{R=S\times T\in \mathcal R}\phi_{S\times T}m_Sf_Th_S\frac{\chi_T}{|T|}$$ where if $R=R_1\times \cdots\times R_N$, $S=S_1\times\cdots\times S_{N_1}\in \D^{N_1}$ and $T=T_1\times\cdots\times T_{N_2}\in \D^{N_2}$, $0<N_1,N_2<N$, $N_1+N_2=N$, then $R_j=S_j$ if $\beta_j=0$ and $R_j=T_j$ if $\beta_j=1$.
\section{Proof of Theorem \ref{main}.}
We start this section by observing the following which is probably well known.
\begin{lemma}\label{lem:inftyBMO}
Let $\phi\in L^2(\T^N)$. The paraproduct $\Pi_\phi:L^\infty(\T^N)\rightarrow \BMO^d(\T^N)$ boundedly if and only if $\phi\in \BMO^d(\T^N)$. Moreover,
$$\|\Pi_\phi\|_{L^\infty(\T^N)\rightarrow \BMO^d(\T^N)}\backsimeq \|\phi\|_{\BMO^d(\T^N)}.$$
\end{lemma}
\begin{proof}
Clearly, if $\Pi_\phi$ is bounded from $L^\infty(\T^N)$ to $\BMO^d(\T^N)$, then we have for any open set $\Omega\subset \T^N$,
$$\|P_\Omega \phi\|_2^2=\|P_\Omega\left(\Pi_\phi 1\right)\|_2^2\le |\Omega|\|\Pi_\phi\|_{L^\infty(\T^N)\rightarrow\BMO^d(\T^N)}$$
which proves that $\phi\in \BMO^d(\T^N)$.

Conversely, assuming that $\phi\in \BMO^d(\T^N)$, we obtain that for any open set $\Omega\subset \T^N$ and any $b\in L^\infty(\T^N)$,
$$\frac{1}{|\Omega|}\sum_{R\in \mathcal {R}, R\subset \Omega}|\phi_R|^2\left(m_Rb\right)^2\le \|b\|_{L^\infty(\T^N)}^2\frac{1}{|\Omega|}\sum_{R\in \mathcal {R}, R\subset \Omega}|\phi_R|^2\le \|b\|_{L^\infty(\T^N)}^2\|\phi\|_{\BMO^d(\T^N)}^2.$$
This proves that $\Pi_\phi b\in \BMO^d(\T^N)$ for any $b\in L^\infty(\T^N)$. The proof of the lemma is complete.
\end{proof}

We next have the following necessary condition.
\begin{proposition} Let $\phi\in L^2(\T^N)$. If $\phi \in \mathcal M(\bmo(\T^N),\BMO^d(\T^N))$, then $\phi\in \BMO_m^d(\T^N)$.
\end{proposition}
\begin{proof}
From the definition of $\BMO_m(\T^N)$ and Lemma \ref{lem:inftyBMO}, it is enough to prove that if $\phi \in \mathcal M(L^\infty(\T^N),\BMO^d(\T^N))$, then for any dyadic interval $I$, $m_I\phi \in \mathcal M(L^\infty(\T^{N-1}),\BMO^d(\T^{N-1}))$.

Let $b\in L^\infty(\T^{N-1})$, and $I\in \D$. Observe that the function $|I|^{1/2}h_I$ is uniformly bounded while $|I|^{-1/2}h_I$ is uniformly in $H^{1,d}(\T)$. We easily obtain
\Beas
\|bm_I\phi\|_{\BMO^d(\T^{N-1})} &=& \sup_{g\in H^{1,d}(\T^{N-1}), \|g\|_{H^{1,d}(\T^{N-1})}\le 1}\left|\int_{\T^{N-1}}m_I\phi(s)b(s) g(s)ds\right|\\ &=&  \sup_{g\in H^{1,d}(\T^{N-1}), \|g\|_{H^{1,d}(\T^{N-1})}\le 1}\left|\int_{\T^{N}}\frac{\chi_I(t)}{|I|}\phi(s,t)b(s)g(s)dsdt\right|\\ &=& \sup_{g\in H^{1,d}(\T^{N-1}), \|g\|_{H^{1,d}(\T^{N-1})}\le 1}\left|\int_{\T^{N}}\phi(s,t)\left(|I|^{1/2}h_I(t)b(s)\right)\left(|I|^{-1/2}h_I(t)g(s)\right)dsdt\right|\\ &\le& \sup_{g\in H^{1,d}(\T^{N-1}), \|g\|_{H^{1,d}(\T^{N-1})}\le 1}\|\phi \left(|I|^{1/2}h_Ib\right)\|_{\BMO^d(\T^N)}\||I|^{-1/2}h_Ig\|_{H^{1,d}(\T^N)}\\ &\le& \sup_{g\in H^{1,d}(\T^{N-1}), \|g\|_{H^{1,d}(\T^{N-1})}\le 1}\|M_\phi\|_{L^\infty(\T^N)\rightarrow \BMO^d(\T^N)}\||I|^{1/2}h_Ib\|_{L^\infty(\T^N)}\|g\|_{H^{1,d}(\T^{N-1})}\\ &\le& \|M_\phi\|_{L^\infty(\T^N)\rightarrow \BMO^d(\T^N)}\|b\|_{L^\infty(\T^{N-1})}.
\Eeas
That is $m_I\phi \in \mathcal M(L^\infty(\T^{N-1}),\BMO^d(\T^{N-1}))$ and the proof is complete.
\end{proof}
\begin{theorem} \label{thm:forw}
 Let $\phi \in \mathcal M(\bmo(\T^N),\BMO^d(\T^N))$. Then $\phi \in \LMO^d(\T^N)\cap L^\infty(\T^N)$.
\end{theorem}
\begin{proof}
Note that if $g\in \BMO^d(\T^N)$, then for any rectangle $R\in \D^N$, we have from the definition of $\BMO^d(\T^N)$ that $$\frac{1}{|R|^{1/2}}|g_R|\le \|g\|_{\BMO^d(\T^N)}.$$
It follows that if $\phi \in \mathcal M(\bmo(\T^N),\BMO^d(\T^N))$, then for any $R\in \D^N$ and any $f\in \bmo(\T^N)$,
\begin{equation}\label{eq:inftynecessity}
\frac{1}{|R|^{1/2}}|\left(\phi f\right)_R|\le \|\phi f\|_{\BMO^d(\T^N)}<\infty.
\end{equation}
Let us take $f=\chi_{R'}$ with $R'\in \D^N$ such that $R$ is a parent (in each direction) of $R'$. It comes that
$$\frac{1}{2^N}|m_{R'}\phi|= \frac{1}{|R|^{1/2}}|\left(\phi f\right)_R|\le \|\phi f\|_{\BMO^d(\T^N)}<\infty.$$
That is for any $R\in \D^N$, $|m_R\phi|<\infty$, hence by the Lebesgue Differentiation Theorem, $\phi\in L^\infty(\T^N)$.

To show that  $\phi \in \LMO^d(\T^N)$, note that given a dyadic rectangle $R = I_1 \times \dots \times I_N$ with $|I_1|= 2^{-k_1}, \dots,
|I_N|= 2^{-k_N}$,
we can construct a "dyadic logarithm" $\log_R=\sum_{j=1}^N\log_{I_j}$ where $\log_{I_j}(t)$ is provided in Lemma \ref{lem:testfunction} with the properties (by the same lemma)
 that $\log_R \in \bmo(\T^N)$, and $\log_R(t) \equiv k_1+ \cdots +k_N$ for $t\in R$ and
 $\|\log_R\|_{\bmo(\T^N)} \le C_N$, where $C_N$ is a constant that does not depend on $R$.

It follows that for any open set $\Omega \subset R$, we have
$$
    \|P_{\Omega} ( \phi \log_R)\|_{L^2(\T^N)}^2 = \left(k_1+\dots +k_N\right)^2  \|P_{\Omega} \phi\|_{L^2(\T^N)}^2.
$$
Consequently, $\|\phi\|_{\LMO^d(\T^N)}^2 \le C \|\phi\|^2_{\mathcal M(\bmo(\T^{N}), \BMO^d(\T^{N}))}$.
\end{proof}


We can now recall a result on the boundedness of the main paraproduct from $\bmo^d(\T^N)$ to $\BMO^d(\T^N)$ obtained in \cite{ps,sehba} and the boundedness of $\Delta_\phi$ on $\BMO^d(\T^N)$ from \cite{bp}.

\begin{theorem}   \label{thm:multothers}
Let $\phi \in L_0^2(\T^N)$, $\vec {\beta}=(\beta_1,\cdots,\beta_n)$, $\beta_j\in \{0,1\}$. Then

\begin{enumerate}
\item $\Pi^{(0, \dots,0)}_\phi = \Pi_\phi:\bmo^d(\T^N) \rightarrow \BMO^d(\T^N)$ is bounded, if and only if $ \phi \in \LMO^d(\T^N)$.
Moreover, $$\|\Pi_\phi\|_{\bmo^d(\T^N) \rightarrow \BMO^d(\T^N)} \approx\|\phi\|_{\LMO^d(\T^N)}.$$
\item $\Pi^{(1,\dots,1)}_\phi = \Delta_\phi: \BMO^d(\T^N) \rightarrow \BMO^d(\T^N)$ is bounded, if and only if $ \phi \in \BMO^d(\T^N)$.
Moreover, $$\|\Delta_\phi\|_{\BMO^d(\T^N) \rightarrow \BMO^d(\T^N)} \approx\|\phi\|_{\BMO^d(\T^N)}.$$
\end{enumerate}

\end{theorem}
Let us observe the following about the boundedness of the other paraproducts.
\begin{proposition}\label{prop:multiothers}
For $\vec {\beta}\neq \vec {\bf 0}=(0,\cdots,0),  \vec {\bf 1}=(1,\cdots,1)$,\\
$\Pi^{\vec {\beta}}_\phi:\BMO_m^d(\T^N) \rightarrow \BMO^d(\T^N)$ is bounded if $ \phi \in \BMO^d(\T^N)$.
Moreover, $$\|\Pi^{\vec {\beta}}_\phi\|_{\BMO_m^d(\T^N) \rightarrow \BMO^d(\T^N)}
  \lesssim \|\phi\|_{\BMO^d(\T^N)}.$$
\end{proposition}
\begin{proof} We suppose that $\phi\in \BMO^d(\T^N)$. Let $b\in \BMO_m^d(\T^N)$, and $\Omega$ an open set in $\T^N$. For any rectangle $T\in \D^{N_2}$, $0<N_2<N$, we write $\Omega_T:=\cup_{S\in \D^{N_1}, S\times T\subseteq \Omega}S$, $N_1+N_2=N$. We easily obtain
\Beas
\|P_\Omega\Pi_\phi^{\vec {\beta}}b\|_{L^2(\T^N)}^2 &=& \|P_\Omega \sum_{S\in \D^{N_1}}\sum_{T\in \D^{N_2}}\frac{\chi_S}{|S|}h_T\phi_{S\times T}m_Tb_S\|_{L^2(\T^N)}^2\\ &=& \sum_{T\in \D^{N_2}}\| \sum_{S\in \D^{N_1}, S\subseteq \Omega_T}\frac{\chi_S}{|S|}\phi_{S\times T}m_Tb_S\|_{L^2(\T^{N_1})}^2\\ &=& \sum_{T\in \D^{N_2}}\| \Delta_{m_Tb}P_{\Omega_T}\phi_T\|_{L^2(\T^{N_1})}^2\\ &\lesssim& \sum_{T\in \D^{N_2}}\| m_Tb\|_{\BMO^d(\T^{N_1})}^2\| P_{\Omega_T}\phi_T\|_{L^2(\T^{N_1})}^2\\ &\lesssim& \| b\|_{\BMO_m^d(\T^{N})}^2\sum_{T\in \D^{N_2}}\|P_{\Omega_T}\phi_T\|_{L^2(\T^{N_1})}^2\\ &\lesssim& \| b\|_{\BMO_m^d(\T^{N})}^2\|P_{\Omega}\phi\|_{L^2(\T^{N})}^2\\ &\lesssim& |\Omega|\| b\|_{\BMO_m^d(\T^{N})}^2\| \phi\|_{\BMO^d(\T^{N})}^2.
\Eeas
\end{proof}

We can now prove the following.
\begin{theorem}  \label{thm:maindy}
   Let $\phi \in L^2(\T^N)$. Then $M_\phi$ defines a bounded operator from $\bmo^d(\T^N)$ to $\BMO^d(\T^N)$, if and only if
   $ \phi \in L^\infty(\T^N)\cap \LMO^d(\T^N)$.
\end{theorem}
\begin{proof} The forward direction follows immediately from Theorem \ref{thm:forw}. For the reverse direction, we decompose
the multiplication operator $M_\phi: \bmo^d(\T^N) \rightarrow \BMO^d(\T^N)$ into a sum of dyadic operators as given in (\ref{eq:expanmuloper}):
$$M_\phi =\sum_{(\veps,\vda,\vba)\in \mathcal U} B_{\veps,\vda,\vba}(\phi,\cdot).$$
We next prove the boundedness of each individual operator $B_{\veps,\vda,\vba}(\phi,\cdot)$. The case $\veps=\vec {0}$ is handled by Theorem \ref{thm:multothers} and Proposition \ref{prop:multiothers}.  The other operators are of the following forms:
$$T_{\phi}^1f(t):=\sum_{R\in \RR}(m_R\phi )f_R h_R(t),\,\,\,t\in \T^N;$$
$$T_{\phi}^2f(t):=\sum_{S\in \D^{N_1}, T\in \D^{N_2}}\left(m_T\phi_S\right) f_{S\times T}\frac{\chi_S(t_S)}{|S|} h_T(t_T),\,\,\,t=(t_S,t_T)\in \T^{N_1}\times \T^{N_2}=\T^N;$$
$$T_{\phi}^3f(t):=\sum_{S\in \D^{N_1}, T\in \D^{N_2}}\left(m_T\phi_S\right) \left(m_Sf_ T\right)h_S(t_S) h_T(t_T),\,\,\,t=(t_S,t_T)\in \T^{N_1}\times \T^{N_2}=\T^N$$
and
$$T_{\phi}^4f(t):=\sum_{R \subseteq \T^{N_1}, S \subseteq \T^{N_2}, T \subseteq \T^{N_3}}
        h_R(t_{J_1}) \frac{\chi_S(t_{J_2})}{|S|} h_T(t_{J_3}) m_T \phi_{S \times R} m_R b_{S \times T},$$
        $t=(t_{J_1},t_{J_2},t_{J_3})\in \T^{N_1}\times \T^{N_2}\times \T^{N_3}=\T^N.$

\begin{lemma} \label{lemm:notbad1}
Let $\phi \in L^2(\T^N)$. Then
 $T_{\phi}^1$
extends as a bounded linear operator from $\bmo^d(\T^N)$ to $\BMO^d(\T^N)$ if and only if  $\phi \in L^\infty(\T^N)$, and
$\| T_\phi^1\|_{\bmo^d(\T^N) \rightarrow \BMO^d(\T^N)} \approx \|\phi \|_{L^\infty(\T^N)}$.
\end{lemma}
\begin{proof}
Note that the operator $T_\phi^1$ is a Haar multiplier with the sequence $\{m_T \phi\}_{ T \in \D^N}$, hence it is bounded on $\BMO^d(\T^N)$ if and only if $\phi \in L^\infty(\T^N)$ with $\| T_\phi^1\|_{\BMO^d (\T^N)\rightarrow \BMO^d(\T^N)} \approx \|\phi \|_{L^\infty(\T^N)}$. Thus we only have to prove that the condition $\phi \in L^\infty(\T^N)$ is necessary. Suppose that $T_\phi^1$ is bounded from $\bmo^d(\T^N)$ to $\BMO(\T^N)$. Then for any dyadic rectangle $R\in \T^N$, and any $f\in \bmo^d(\T^N)$, we have $$\frac{1}{|R|^{1/2}}|\left(\phi f\right)_R|\le \|T_\phi^1f\|_{\BMO^d(\T^N)}\le \|T_\phi^1\|_{\bmo^d(\T^N)\rightarrow \BMO^d(\T^N)}\|f\|_{\bmo^d(\T^N)}.$$
Taking $f(t)=|R|^{1/2}h_R(t)$ which is uniformly in $\bmo^d(\T^N)$, we obtain that for any dyadic rectangle $R\in \T^N$,
$$|m_R\phi|\le \|T_\phi^1\|_{\bmo^d(\T^N)\rightarrow \BMO^d(\T^N)}.$$
Hence $\|\phi \|_{L^\infty(\T^N)}\le \|T_\phi^1\|_{\bmo^d(\T^N)\rightarrow \BMO^d(\T^N)}$.
\end{proof}

We now estimate the operators $T_\phi^2$.
\begin{lemma} \label{lemm:notbad2}
Let $\phi \in \BMO_m^d(\T^N)$.
Then the operator $T_\phi^2$  defines a bounded linear operator on $\BMO^d(\T^N)$ with
$\| T_\phi^2\|_{\BMO^d(\T^N) \rightarrow \BMO^d(\T^N)} \lesssim \|\phi \|_{\BMO_m^d(\T^N)}$.
\end{lemma}
\begin{proof} We first prove that if $\phi \in \BMO_m^d(\T^N)$, then $T_\phi^2$ is bounded on $L^2(\T^N)$. Let $f\in L^2(\T^N)$. Then
\begin{eqnarray*}
   \| T_{\phi}^2f\|_{L^2(\T^N)}^2
   &=& \|\sum_{S\in \D^{N_1}, K\in \D^{N_2}}
        h_K(t_{K}) \frac{\chi_S(t_{S})}{|S|}  (m_K \phi_{S}) f_{K \times S}\|_{L^2(\T^N)}^2 \\
  &\le&  \sum_{K\in \D^{N_2}}\|\sum_{S\in \D^{N_1}}\frac{\chi_S(t_{S})}{|S|}  (m_K \phi_{S}) f_{K \times S}\|_{L^2(\T^{N_1})}^2\\
  &=& \sum_{K\in \D^{N_2}}\|\Delta_{m_K\phi}f_K\|_{L^2(\T^{N_1})}^2\\ &\le&
  \sum_{K\in \D^{N_2}}\|m_K\phi\|_{\BMO^d(\T^{N_1})}^2\|f_K\|_{L^2(\T^{N_1})}^2\\ &\le&
  \|\phi\|_{\BMO_m^d(\T^N)}^2\sum_{K\in \D^{N_2}}\|f_K\|_{L^2(\T^{N_1})}^2\\ &=& \|\phi\|_{\BMO_m^d(\T^N)}^2\|f\|_{L^2(\T^{N})}^2.
\end{eqnarray*}
Next for $f \in \BMO^d(\T^N)$ and for $\Omega \subseteq \T^N$ open, we obtain from the above estimate
\begin{eqnarray*}
   \| P_{\Omega}  T_{\phi}^2f\|_{L^2(\T^{N})}^2
   &=& \|P_{\Omega} \sum_{S\in \D^{N_1}, K\in \D^{N_2}}
        h_K(t_{K}) \frac{\chi_S(t_{S})}{|S|}  (m_K \phi_{S}) f_{K \times S}\|_{L^2(\T^{N})}^2 \\
   &=& \|\sum_{S\in \D^{N_1}, K\in \D^{N_2}, K \times S \subseteq \Omega}
        h_K(t_{K}) \frac{\chi_S(t_{S})}{|S|}  (m_K \phi_{S}) f_{K \times S}\|_{L^2(\T^{N})}^2 \\ &=& \|   T_{\phi}^2P_\Omega f\|_{L^2(\T^{N})}^2\\ &\lesssim&  \|\phi\|_{\BMO_m^d(\T^N)}^2\|P_\Omega f\|_{L^2(\T^{N})}^2\\ &\le& |\Omega|\|\phi\|_{\BMO_m^d(\T^N)}^2\|f\|_{\BMO^d(\T^{N})}^2.
\end{eqnarray*}

Thus $\| T_\phi^2\|_{\BMO^d(\T^N) \rightarrow \BMO^d(\T^N)} \lesssim \|\phi \|_{\BMO_m^d(\T^N)}$. The proof is complete.
\end{proof}

Let us recall that $T_\phi^3$ is given by
$$T_{\phi}^3f(t):=\sum_{S\in \D^{N_1}, K\in \D^{N_2}}\left(m_K\phi_S\right) \left(m_Sf_ K\right)h_S(t_S) h_K(t_K),\,\,\,t=(t_S,t_K)\in \T^{N_1}\times \T^{N_2}=\T^N.$$ 
We have the following result whcih provides an equivalent definition of the mean $BMO$ space.
\begin{proposition}\label{prop:notbad3}
Let $\phi\in L^2(\T^N)$. Then the operator $T_\phi^3$ is bounded from $\bmo^d(\T^N)$ to $\BMO^d(\T^N)$ if and only if $\phi\in \BMO_m^d(\T^N)$. Moreover, $$\|T_\phi^3\|_{\bmo^d(\T^N)\rightarrow \BMO^d(\T^N)}\backsimeq \|\phi\|_{\BMO_m^d(\T^N)}.$$
\end{proposition}
\begin{proof}
We first suppose that the operator $T_\phi^3$ is bounded from $\bmo^d(\T^N)$ to $\BMO^d(\T^N)$. Let $N_2\in \{1,2,\cdots,N-1\}$ be fixed. For $R\in \mathcal D^{N_2}$, consider the function
$$f(t,s)==|R|^{1/2}h_R(s),\,\,\, s\in \T^ {N_2},\,\,\,t\in \T^{N-{N_2}}.$$ It is clear that the function $f$ is a uniformly bounded function so it belongs to $\bmo^d(\T^N)$.
\medskip
We suppose that our given $R$ is a rectangle in the last ${N_2}$ variables. We can also suppose without loss of generality that in our operators $T_\phi^3$, $S$ is a rectangle in the first $N_1$ variables and $K$ a rectangle in the last $N_2$ variables. We easily obtain that
$$T_\phi^3f(t,s)=\sum_{S\in \D^{N_1}}|R|^{1/2}\left(m_R\phi_S\right) h_S(t) h_R(s).$$
It follows that for any open set $\Omega\subset \T^{N_1}$,
\Beas
|R|\|P_\Omega(m_R\phi)\|_{L^2(\T^{N_1})}^2 &=& \|P_{R\times\Omega}(T_\phi^3\,f)\|_{L^2(\T^{N})}^2\\ &\le& |R\times \Omega|\|T_\phi^3f\|_{BMO^d(\T^N)}^2\\ &\lesssim& |R\times \Omega|\|f\|_{bmo^d(\T^N)}^2\lesssim |R\times \Omega|.
\Eeas
That is for any open set $\Omega\subset \T^{N_1}$,
$$\|P_\Omega(m_R\phi)\|_{L^2(\T^{N_1})}^2\lesssim |\Omega|.$$
Hence for any $1\le N_2<N$, for all $R\in \mathcal D^{N_2}$, $m_R\phi\in BMO^d(\T^{N-N_2})$. Thus $\phi\in BMO_m^d(\T^N)$.
\vskip .2cm
We next suppose that $\phi\in \BMO_m^d(\T^N)$ and prove that each $T_\phi^3$ is bounded from $\bmo^d(\T^N)$ to $\BMO^d(\T^N)$. That is we check that given $\phi$ as above and for any $b\in \bmo^d(\T^N)$, $\Pi_{T_\phi^3b}:L^2(\T^N)\rightarrow L^2(\T^N)$ boundedly, that is there exists $C>0$ such that for any $f\in L^2(\T^N)$,
\begin{equation}\label{eq:strongnotbad3}\|F:=\Pi_{T_\phi^3b}f\|_{L^2(\T^N)}\le C\|f\|_{L^2(\T^N)}.\end{equation}
\vskip .2cm
For $k\in \mathbb N$, and for $f\in L^2(\T^N)$, define $F_k:=\Pi_{T_\phi^3b}E_k^{(1)}f$. Then one has that for any $g\in L^2(\T^N)$,
$$\lim_{k\rightarrow \infty}\int_{\T^N}F_k(t)g(t)dt=\int_{\T^N}F(t)g(t)dt.$$
It follows that to obtain (\ref{eq:strongnotbad3}), it is enough to prove that for any $k\in \mathbb N$,
$$\|\Pi_{T_\phi^3b}E_k^{(1)}f\|_{L^2(\T^N)}\le C\|\phi\|_{\BMO_m^d(\T^{N})}\|b\|_{\bmo^d(\T^{N})}\|f\|_{L^2(\T^N)}$$
where the constant $C>0$ does not depend on $k$.
That is $$\|\Pi_{T_\phi^3b}E_k^{(1)}\|_{L^2(\T^N)\rightarrow L^2(\T^N)}\le C\|\phi\|_{\BMO_m^d(\T^{N})}\|b\|_{\bmo^d(\T^{N})}.$$
But following \cite{ps}, we have $$\|\Pi_{T_\phi^3b}E_k^{(1)}\|_{L^2(\T^N)\rightarrow L^2(\T^N)}=\|\Pi_{\widetilde {T_\phi^3b}}\|_{L^2(\T^N)\rightarrow L^2(\T^N)}$$
where for $g\in L^2(\T^N)$, $\tilde {g}$ is defined as follows:
 for $R=\prod_{j=1}^NR_j=R_1\times S\in \D^N$, $S=R_2\times\cdots\times R_N\in \D^{N-1}$,
$$
   (\tilde {g})_{R} =\left\{ \begin{matrix} g_{R} & \text{ if } & |R_1|
   > 2^{-k}\\
 (\sum_{R_1' \subseteq R_1} |g_{R_1'\times S}|^2)^{1/2}   & \text{ if } &
   |R_1| = 2^{-k}\\
      0 & \text{ otherwise. }&
                               \end{matrix} \right.
$$
Hence, it is enough to prove that $\widetilde {T_\phi^3b}$ is uniformly in $\BMO^d(\T^N)$, that is its $\BMO$-norm is bounded by some constant that does not depend on $k$.
\vskip .2cm
Let us write $E_k$ and $Q_k$ for $E_k^{(1)}$ and $Q_k^{(1)}$ respectively. Clearly
\begin{equation*}
  \widetilde {T_\phi^3b}  =  \widetilde {T_{E_k\phi}^3b} + \widetilde {T_{Q_k\phi}^3b}=
     I + II.
\end{equation*}

Let us first consider the second term $II$. We observe that $\widetilde {T_{Q_k\phi}^3b}$ has only nontrivial Haar coefficients for $R=R_1\times\cdots\times R_N \in \RR$, with $|R_1|=2^{-k}$. We first compute $\|P_R \widetilde {T_{Q_k\phi}^3b}\|_{L^2(\T^N)}^2$
for $R\in \RR$ a rectangle of this type. We write $R=K\times L$, $K\in \D^{N_1}$ and $L\in \D^{N_2}$.
\begin{eqnarray*}
  \int_R |P_R  \widetilde {T_{Q_k\phi}^3b}  |^2 ds dt
&\le& \sum_{S \subseteq K}\sum_{T\subseteq L}
  |m_S\phi_{T}|^2 |m_{T} b_S|^2 \\
&=&  \sum_{S\subseteq K}\|\Pi_{P_L m_S\phi} \chi_L b_S\|^2_{L^2(\T^{N_2})} \\
& \lesssim &   \sum_{S\subseteq K}\|P_L m_S\phi\|^2_{\BMO^d(\T^{N_2})} \| \chi_L b_S\|_{L^2(\T^{N_2})}^2 \\
&\lesssim &    \| \phi\|^2_{\BMO_m^d(\T^N)} \sum_{S\subseteq K} \| \chi_L b_S\|_{L^2(\T^{N_2})}^2\\ &=& \| \phi\|^2_{\BMO_m^d(\T^N)}\| \chi_L P_K b\|_{L^2(\T^{N})}^2\\ &\lesssim& |K||L|\| \phi\|^2_{\BMO_m^d(\T^N)}\| b\|^2_{\bmo^d(\T^N)}
\end{eqnarray*}
where we use Lemma \ref{lem:estim} at the last inequality.

Now for $\Omega\subseteq \T^N$ open, we denote by $\mathcal {M}_k(\Omega)$ the set of all rectangles $R=R_1\times\cdots\times R_N\in \RR,\,\,\, R\subseteq \Omega$ which are maximal in $\Omega$ with respect to $|R_1|=2^{-k}$.
If $\mathcal {M}_k(\Omega)=\emptyset$, then
$$\|P_\Omega \widetilde {T_{Q_k\phi}^3b}\|_{L^2(\T^N)}^2=0\le |\Omega|\| \phi\|^2_{\BMO_m^d(\T^N)} \|b\|^2_{\bmo^d(\T^N)}.$$

If $\mathcal {M}_k(\Omega)\neq \emptyset$, then using the above estimate of the $L^2-$norm of $P_R \widetilde {T_{Q_k\phi}^3b}$ for $R=R_1\times\cdots\times R_N \in \RR$, with $|R_1|=2^{-k}$, we obtain
\begin{eqnarray*}
\|P_\Omega \widetilde {T_{Q_k\phi}^3b}\|_{L^2(\T^N)}^2 &=&  \sum_{R\in \mathcal {M}_k(\Omega)}\sum_{R'\subseteq R, |R'_1|=2^{-k}}|\left(\widetilde {T_{Q_k\phi}^3b}\right)_{R'}|^2\\ &\le& \sum_{R\in \mathcal {M}_k(\Omega)}\|P_R \widetilde {T_{Q_k\phi}^3b}\|_{L^2(\T^N)}^2\\ &\lesssim& \| \phi\|^2_{\BMO_m^d(\T^N)} \|b\|^2_{\bmo^d(\T^N)}\sum_{R\in \mathcal {M}_k(\Omega)}|R|\\ &\lesssim&  |\Omega|\| \phi\|^2_{\BMO_m^d(\T^N)} \|b\|^2_{\bmo^d(\T^N)}.
\end{eqnarray*}
\vskip .2cm
Let us now consider term I.
For any open set $\Omega \subseteq \T^N$, we write for $T\in \D^{N_2}$ given ,
 $$\Omega_T=\cup_{S\in \D^{N_1}, S\times T\subseteq \Omega}S.$$ We obtain using properties of $\bmo^d(\T^N)$ and the estimate in Lemma \ref{lem:estim},
\begin{eqnarray*}
    \frac{1}{|\Omega|} \|P_{\Omega} \widetilde {T_{E_k\phi}^3b}\|_{L^2(\T^N)}^2
&\le& \frac{1}{|\Omega|} \sum_{R =S\times T\in \D^{N_1}\times \D^{N_2}, |T_1|>2^{-k},
   R \subset \Omega}  |m_{S}\phi_T|^2 |m_T b_S|^2 \\
& \lesssim & \frac{1}{|\Omega|} \sum_{S\in \D^{N_1}}  \|P_{\Omega_S}\Pi_{E_km_S \phi}\chi_{\Omega_S}b_S\|_{L^2(\T^{N_2})}^2 \\
&\lesssim &  \frac{1}{|\Omega|} \sum_{S\in \D^{N_1}}  \|\Pi_{P_{\Omega_S}E_km_S \phi}\left(\sum_{T\in \D^{N_2}, |T_1|=2^{-k}, S\times T\subseteq \Omega}\chi_T\chi_{\Omega_S}b_S\right)\|_{L^2(\T^{N_2})}^2\\ &\lesssim&  \frac{1}{|\Omega|} \sum_{S\in \D^{N_1}}\| P_{\Omega_S}E_km_S\phi \|^2_{\BMO^d(\T^{N_2})}\|\sum_{T\in \D^{N_2}, |T_1|=2^{-k}, S\times T\subseteq \Omega}\chi_Tb_S\|_{L^2(\T^{N_2})}^2\\ &\lesssim&  \| \phi \|^2_{\BMO_m^d(\T^N)}\frac{1}{|\Omega|} \sum_{S\in \D^{N_1}}\|\sum_{T\in \D^{N_2}, |T_1|=2^{-k}, S\times T\subseteq \Omega}\chi_Tb_S\|_{L^2(\T^{N_2})}^2\\ &\lesssim&  \| \phi \|^2_{\BMO_m^d(\T^N)}\frac{1}{|\Omega|} \|\sum_{T\in \D^{N_2}, |T_1|=2^{-k}}\chi_T\sum_{S\in \D^{N_1},S\times T\subseteq \Omega}b_Sh_S\|_{L^2(\T^{N})}^2\\ &\lesssim&  \| \phi \|^2_{\BMO_m^d(\T^N)}\frac{1}{|\Omega|} \sum_{T\in \D^{N_2}, |T_1|=2^{-k}}\|\chi_TP_{\Omega_T}b\|_{L^2(\T^{N})}\\ &\lesssim&  \| \phi \|^2_{\BMO_m^d(\T^N)}\|b\|^2_{\bmo^d(\T^N)}\frac{1}{|\Omega|}\sum_{T\in \D^{N_2}, |T_1|=2^{-k}}|T||\Omega_T|\\ &\lesssim& \| \phi \|^2_{\BMO^d(\T^N)}  \|b\|^2_{\bmo^d(\T^N)}.
\end{eqnarray*}
We conclude that for any $\phi\in \BMO_m^d(\T^N)$ and any $b\in \bmo^d(\T^N)$, and for any $k\in \mathbb N_0$,
$$\|\Pi_{T_\phi^3b}E_k^{(1)}\|_{L^2(\T^N)\rightarrow L^2(\T^N)}\le C\|\phi\|_{\BMO_m^d(\T^{N})}\|b\|_{\bmo^d(\T^{N})}$$
and the constant $C>0$ does not depend on $k$. Thus
$$\|\Pi_{T_\phi^3b}\|_{L^2(\T^N)\rightarrow L^2(\T^N)}\le C\|\phi\|_{\BMO_m^d(\T^{N})}\|b\|_{\bmo^d(\T^{N})}$$
which is equivalent to saying that $T_\phi^3b\in \BMO^d(\T^N)$.
The proof is complete.
\end{proof}

Let us prove that the last type of operators, $T_\phi^4$ is also bounded from $\bmo^d(\T^N)$ to $\BMO^d(\T^N)$. We recall that
$$T_{\phi}^4f(t):=\sum_{R \in \D^{N_1}, S \in \D^{N_2}, T \in \D^{N_3}}
        h_R(t_{N_1}) \frac{\chi_S(t_{N_2})}{|S|} h_T(t_{N_3}) m_T \phi_{S \times R} m_R b_{S \times T},$$
        $t=(t_{N_1},t_{N_2},t_{N_3})\in \T^{N_1}\times \T^{N_2}\times \T^{N_3}.$
\begin{proposition}\label{prop:notbad4}
Let $\phi\in L^2(\T^N)$. Then the operator $T_\phi^4$ is a bounded operator from $\bmo^d(\T^N)$ to $\BMO^d(\T^N)$ if $\phi\in \BMO_m^d(\T^N)$. Moreover,
$$\|T_\phi^4\|_{\bmo^d(\T^N)\rightarrow \BMO^d(\T^N)}\lesssim \|\phi\|_{\BMO_m^d(\T^N)}.$$
\end{proposition}
\begin{proof}
Before proving the proposition, let us observe the following boundedness of the paraproducts $\Pi^{\vec {\beta}}$ on $L^2(\T^N)$.
\begin{lemma}\label{lem:othersparaL2}
If $\phi\in \BMO^d(\T^N)$, then for any $\vec {\beta}\in \{0,1\}^{N}$, $\Pi_\phi^{\vec {\beta}}$ is bounded on $L^2(\T^N)$. Moreover,
$$\|\Pi_\phi^{\vec {\beta}}\|_{L^2(\T^N)\rightarrow L^2(\T^N)}\lesssim \|\phi\|_{\BMO^d(\T^N)}\backsimeq \|\Pi_\phi\|_{L^2(\T^N)\rightarrow L^2(\T^N)}.$$
\end{lemma}
\begin{proof}
The proof is known for $\Pi=\Pi^{\vec 0}$ and $\Pi^{\vec 1}$ (see \cite{bp}). Let us prove the lemma for $\vec {\beta}\neq \vec {0}, \vec {1}$. We will follow the two-parameter proof of \cite{pottsad}. Recall that $\Pi^{\vec \beta}$ is of the following general form
$$\Pi_\phi^{\vba}f:=\sum_{S\in \D^{N_1}, T\in \D^{N_2}}\phi_{S\times T}m_Sf_Th_S\frac{\chi_T}{|T|},\,\,\, 0<N_1,N_2<N,\,\,\,N_1+N_2=N.$$
It is clear given $\phi\in \BMO^d(\T^N)$, the boundedness of $\Pi_\phi^{\vba}$ on $L^2(\T^N)$ is equivalent to saying that for any $f,g\in L^2(\T^N)$, the quantity $K(f,g):=\sum_{S\in \D^{N_1}, T\in \D^{N_2}}m_Sf_Tm_Tg_Sh_Sh_T$ is in $H^{1,d}(\T^N)$.

Let us prove that $K:L^2(\T^N)\times L^2(\T^N)\rightarrow H^{1,d}(\T^N)$ boundedly. We write $$\tilde {f}_T(s)=\sup_{S\in \D^{N_1}}\frac{\chi_S(s)}{|S|}\int_S |f_T(t)|dt.$$ It comes that
\Beas
\|K(f,g)\|_{H^{1,d}(\T^N)} &:=& \int_{\T^{N_1}}\int_{\T^{N_2}}\left(\sum_{S\in \D^{N_1}, T\in \D^{N_2}}\frac{\chi_{S\times T}(s,t)}{|S||T|}|m_Sf_T|^2|m_Tg_S|^2\right)^{1/2}dsdt\\ &\le& \int_{\T^{N_1}}\int_{\T^{N_2}}\left(\sum_{S\in \D^{N_1}, T\in \D^{N_2}}\frac{\chi_{S\times T}(s,t)}{|S||T|}\left(\tilde {f}_T(s)\right)^2\left(\tilde {g}_S(t)\right)^2\right)^{1/2}dsdt\\ &\le& \int_{\T^{N_1}}\int_{\T^{N_2}}\left(\sum_{ T\in \D^{N_2}}\frac{\chi_{T}(t)}{|T|}|\tilde {f}_T(s)|^2\right)^{1/2}\left(\sum_{S\in \D^{N_1}}\frac{\chi_{S}(s)}{|S|}|\tilde {g}_S(t)|^2\right)^{1/2}dsdt\\ &\le& \left(\sum_{ T\in \D^{N_2}}\|\tilde {f}_T\|_{L^2(\T^{N_1})}^2\right)^{1/2}\left(\sum_{ S\in \D^{N_1}}\|\tilde {g}_S\|_{L^2(\T^{N_2})}^2\right)^{1/2}\\ &\lesssim& \left(\sum_{ T\in \D^{N_2}}\|f_T\|_{L^2(\T^{N_1})}^2\right)^{1/2}\left(\sum_{ S\in \D^{N_1}}\|g_S\|_{L^2(\T^{N_2})}^2\right)^{1/2}\\ &\lesssim& \|f\|_{L^2(\T^N)} \|g\|_{L^2(\T^N)}.
\Eeas
It follows easily that $$\|\Pi_\phi^{\vec {\beta}}\|_{L^2(\T^N)\rightarrow L^2(\T^N)}\lesssim \|\phi\|_{\BMO^d(\T^N)}\backsimeq \|\Pi_\phi\|_{L^2(\T^N)\rightarrow L^2(\T^N)}.$$
\end{proof}
Coming back to the proof of the proposition, we proceed as in the previous proposition. We estimate the $BMO$-norm of $\widetilde {T_\phi^4b}$. As before, we write $E_k$ and $Q_k$ for $E_k^{(1)}$ and $Q_k^{(1)}$ respectively, and
\begin{equation*}
  \widetilde {T_\phi^4b}  =  \widetilde {T_{E_k\phi}^4b} + \widetilde {T_{Q_k\phi}^4b}=
     I + II.
\end{equation*}

Let us start again with the second term $II$. We observe that $\widetilde {T_{Q_k\phi}^4b}$ has only nontrivial Haar
coefficients for $K=K_1\times\cdots\times K_N \in \RR$, with $|K_1|=2^{-k}$. We first compute $\|P_K \widetilde {T_{Q_k\phi}^4b}\|_{L^2(\T^N)}^2$
for $K\in \RR$ a rectangle of this type. We write $K=R\times S\times T$, $R\in \D^{N_1}$, $S\in \D^{N_2}$ and $T\in \D^{N_3}$. It follows using among others Lemma \ref{lem:othersparaL2} and Lemma \ref{lem:estim} that
\begin{eqnarray*}
\|P_K\widetilde {T_{Q_k\phi}^4b}\|_{L^2(\T^N)}^2 &\le& \|\sum_{R' \subseteq R\in \D^{N_1}}\sum_{S'\subseteq S\in \D^{N_2}}\sum_{T'\subseteq T\in\D^{N_3}}
        h_{R'} \frac{\chi_{S'}}{|S'|} h_{T'} m_{T'} \phi_{S' \times R'} m_{R'} b_{S' \times T'}\|_{L^2(\T^N)}^2\\ 
        &\le& \sum_{T'\subseteq T\in\D^{N_3}}\|\sum_{R' \subseteq R\in \D^{N_1}}\sum_{S'\subseteq S\in \D^{N_2}}h_{R'} \frac{\chi_{S'}}{|S'|} m_{T'} \phi_{S' \times R'} m_{R'} b_{S' \times T'}\|_{L^2(\T^{N-N_3})}^2\\ 
        &=& \sum_{T'\subseteq T\in\D^{N_3}}\|\Pi_{P_{S\times R}m_{T'}\phi}^{\vec {\beta}}\chi_RP_Sb_{T'}\|_{L^2(\T^{N-N_3})}^2\\ 
        &\lesssim& \sum_{T'\subseteq T\in\D^{N_3}}\|P_{S\times R}m_{T'}\phi\|_{\BMO^d(\T^{N-N_3})}^2\|\chi_R P_Sb_{T'}\|_{L^2(\T^{N-N_3})}^2\\
        &\lesssim& \|\phi\|_{\BMO_m^d(\T^N)}^2 \sum_{T'\subseteq T\in\D^{N_3}} \|\chi_RP_Sb_{T'}\|_{L^2(\T^{N-N_3})}^2\\  
        &=& \|\phi\|_{\BMO_m^d(\T^N)}^2\|\chi_RP_{S\times T}b\|_{L^2(\T^{N})}^2\\ 
        &\lesssim& |R||S||T|\|\phi\|_{\BMO_m^d(\T^N)}^2\|b\|_{\bmo^d(\T^N)}^2.
\end{eqnarray*}

The case of general open subset of $\T^N$ is then handled as in the proof of the previous proposition.

Next we consider the first term. Given an open set $\Omega\subset \T^N$, for $T\in \D^{N_3},\,\,\,0<N_3<N$, we define by $\Omega_T$ as follows

\Bes
\Omega_T:=\cup_{S\times R\in \D^{N_2}\times\D^{N_1},R\times S\times T\subset \Omega} S\times R,
\,\,\,N_1+N_2+N_3=N.
\Ees
Firstly, using Lemma \ref{lem:othersparaL2} and the definition of the operators $\Pi^{\vec{\beta}}$, we obtain for any $b\in \bmo^d(\T^N)$,
\Beas
\|P_\Omega\left(\widetilde {T_{E_k\phi}^4}b\right)\|_{L^2(\T^N)}^2 &=& \|P_\Omega\left(\sum_{R \in \D^{N_1}, S \in \D^{N_2}, T \in \D^{N_3}, |R_1|>2^{-k}}
        h_R \frac{\chi_S}{|S|} h_T m_T \phi_{S \times R} m_R b_{S \times T}\right)\|_{L^2(\T^N)}^2\\ &\le& \sum_{T \in \D^{N_3}}\|P_{\Omega_T}\sum_{ S \in \D^{N_2}, R \in \D^{N_1}, |R_1|>2^{-k}}\frac{\chi_S}{|S|} h_R m_T \phi_{S \times R} m_R b_{S \times T}\|_{L^2(\T^{N-N_3})}^2\\ &=& \sum_{T \in \D^{N_3}}\|\sum_{ S \in \D^{N_2}, R \in \D^{N_1}, |R_1|>2^{-k}, R\times S\times T\subseteq \Omega}\frac{\chi_S}{|S|} h_R m_T \phi_{S \times R} m_R b_{S \times T}\|_{L^2(\T^{N-N_3})}^2\\ &=& \sum_{T \in \D^{N_3}}\|\Pi_{P_{\Omega_T}E_km_T\phi}^{\vec {\beta}}\left(b_T\right)\|_{L^2(\T^{N-N_3})}^2.
\Eeas 
Secondly, slicing and using Lemma \ref{lem:othersparaL2}, and the properties of $\BMO_m^d(\T^N)$, we obtain for any $b\in \bmo^d(\T^N)$,

\Beas       
  \|P_\Omega\left(\widetilde {T_{E_k\phi}^4}b\right)\|_{L^2(\T^N)}^2 &\leq& \sum_{T \in \D^{N_3}}\|\Pi_{P_{\Omega_T}E_km_T\phi}^{\vec {\beta}}\left(\sum_{ R \in \D^{N_1}, |R_1|=2^{-k}}\chi_R\sum_{ S \in \D^{N_2},  R\times S\times T\subseteq \Omega}h_Sb_{S\times T}\right)\|_{L^2(\T^{N-N_3})}^2\\  &\lesssim& \sum_{T \in \D^{N_3}}\|P_{\Omega_T}E_km_T\phi\|_{\BMO^d(\T^{N-N_3})}^2\|\sum_{ R \in \D^{N_1}, |R_1|=2^{-k}}\chi_R\\ & & \sum_{ S \in \D^{N_2},  R\times S\times T\subseteq \Omega}h_Sb_{S\times T}\|_{L^2(\T^{N-N_3})}^2.
\Eeas
Finally, using Lemma \ref{lem:estim}, we obtain
\Beas
\|P_\Omega\left(\widetilde {T_{E_k\phi}^4}b\right)\|_{L^2(\T^N)}^2   &\lesssim& \|\phi\|_{\BMO_m^d(\T^N)}^2\sum_{T \in \D^{N_3}}\|\sum_{ R \in \D^{N_1}, |R_1|=2^{-k}}\chi_R\sum_{ S \in \D^{N_2},  R\times S\times T\subseteq \Omega}h_Sb_{S\times T}\|_{L^2(\T^{N-N_3})}^2\\ &=& \|\phi\|_{\BMO_m^d(\T^N)}^2\|\sum_{ R \in \D^{N_1}, |R_1|=2^{-k}}\chi_R\sum_{ S \in \D^{N_2}, T\in \D^{N_3}, R\times S\times T\subseteq \Omega}h_Th_Sb_{S\times T}\|_{L^2(\T^N)}^2\\ &\lesssim& \|\phi\|_{\BMO_m^d(\T^N)}^2\sum_{R\in \D^{N_1} , |R_1|=2^{-k}}\|\chi_{R}P_{\Omega_{R}}b\|_{L^2(\T^N)}^2\\ &\lesssim& \|\phi\|_{\BMO_m^d(\T^N)}^2\|b\|_{\bmo^d(\T^N)}^2\sum_{R\times S\in \D^{N-N_3} , |R_1|=2^{-k}}|R||\Omega_{R}|\\
        &\le& |\Omega|\|\phi\|_{\BMO_m^d(\T^N)}^2\|b\|_{\bmo^d(\T^N)}^2.
\Eeas
The proof of this proposition is complete.
\end{proof}
Using the decomposition (\ref{eq:expanmuloper}), the paraproduct result in Theorem \ref{thm:multothers}, Proposition \ref{prop:multiothers}, Lemma \ref{lemm:notbad1} and Lemma \ref{lemm:notbad2}, Proposition \ref{prop:notbad3} and Proposition \ref{prop:notbad4}, we obtain that
\begin{equation}   \label{eq:dyest}
   \| M_\phi f \|_{\BMO^d(\T^N)} \lesssim ( \|\phi\|_\infty + \|\phi \|_{\LMO^d(\T^N)}) \|f\|_{\bmo^d(\T^N)}.
\end{equation}

This finishes the proof of Theorem \ref{thm:maindy}.
\end{proof}

We can now prove Theorem \ref{main}.
\begin{proof}[Proof of Theorem \ref{main}] We recall that $M_\phi$ is the multiplication operator by $\phi$. If $M_\phi$ defines a bounded operator from $\bmo(\T^N)$ to $\BMO(\T^N)$, then because $\BMO(\T^N)=\cap_{\alpha\in \T^N}\BMO^{d,\alpha}(\T^N)$, $\phi\in \mathcal {M}(\bmo(\T^N), \BMO^{d,\alpha}(\T^N))$ for all $\alpha\in \T^N$. It follows from Theorem \ref{thm:forw} that $\phi\in L^\infty(\T^N)\cap \LMO^{d,\alpha}(\T^N)$ for all $\alpha\in \T^N$. Hence by the definition of $\LMO(\T^N)$, we conclude that $\phi\in L^\infty(\T^N)\cap \LMO(\T^N)$.

For the reverse direction, let $\phi \in \LMO(\T^N) \cap L^\infty(\T^N)$ and let $f \in \bmo(\T^N)$. By Definition
\ref{def:LMO} and observations made on $\bmo(\T^N)$ at the introduction, $\phi \in  \LMO^{d,\alpha}(\T^N) \cap L^\infty(\T^N) $ and
$f \in \bmo^{d,\alpha}(\T^N)$ for all $\alpha \in \T^N$, with uniformly bounded respective norms. Thus by Theorem \ref{thm:maindy}, $(M_\phi f) \in \BMO^{d,\alpha}(\T^N)$ for all $\alpha \in \T^N$, with
$$
    \| M_\phi f \|_{\BMO^{d,\alpha}(\T^N)} \lesssim  (\|\phi \|_\infty + \| \phi\|_{\LMO^{d.\alpha}}) \|f\|_{\bmo}.
$$
Consequently,
$$
\|M_\phi f \|_{\BMO} \lesssim  (\|\phi \|_\infty + \| \phi\|_{\LMO} ) \|f\|_{\BMO}.
$$
This finishes the proof of the theorem.
\end{proof}

\bibliographystyle{plain}

\end{document}